\documentclass[11pt]{amsart}
\usepackage[usenames,dvipsnames]{pstricks}
\usepackage[mathscr]{eucal}
\usepackage{amsfonts,amsmath,amssymb,amsthm,amscd,amsxtra}
\usepackage{enumerate,verbatim}
\usepackage[all,2cell,ps]{xy}
\usepackage[notcite, notref]{}
\usepackage[pagebackref]{hyperref}

\theoremstyle{plain} 

\newtheorem{thm}{Theorem}[section]

\newtheorem{prop}[thm]{Proposition}
\newtheorem{cor}[thm]{Corollary}

\newtheorem{lem}[thm]{Lemma}

\theoremstyle{definition}

\newtheorem{obs}[thm]{Observation}

\newtheorem{eg}[thm]{Example}

\newtheorem{ques}[thm]{Question}

\newtheorem{rmk}[thm]{Remark}

\numberwithin{equation}{section}

\newcommand{\fm}{\mathfrak{m}}

\newcommand{\CC}{\mathbb{C}}

\newcommand{\ZZ}{\mathbb{Z}}
\newcommand{\NN}{\mathbb{N}}

\newtheorem{chunk}[thm]{\hspace*{-1.065ex}\bf}

\def\CI{\operatorname{\mathsf{CI-dim}}}
\def\G-dim{\operatorname{\mathsf{G-dim}}}
\def\cx{\operatorname{\mathsf{cx}}}
\def\pd{\operatorname{\mathsf{pd}}}
\def\depth{\operatorname{\mathsf{depth}}}
\def\edim{\operatorname{\mathsf{embdim}}}

\def\Hom{\operatorname{\mathsf{Hom}}}
\def\Tr{\mathsf{Tr}\hspace{0.01in}}
\def\Ext{\operatorname{\mathsf{Ext}}}
\def\Spec{\operatorname{\mathsf{Spec}}}
\def\Tor{\operatorname{\mathsf{Tor}}}
\def\cod{\operatorname{\mathsf{codim}}}
\def\dim{\operatorname{\mathsf{dim}}}

\DeclareMathOperator{\hh}{H}

\def\len{\operatorname{\mathsf{\lambda}}}

\DeclareMathOperator{\rank}{rank}

\DeclareMathOperator{\Supp}{Supp}

\newcommand{\tensor}{\otimes^{\bf L}}

\def\urltilda{\kern -.15em\lower .7ex\hbox{\~{}}\kern .04em}
\def\urldot{\kern -.10em.\kern -.10em}\def\urlhttp{http\kern -.10em\lower -.1ex
\hbox{:}\kern -.12em\lower 0ex\hbox{/}\kern -.18em\lower 0ex\hbox{/}}

\begin{document}
\baselineskip=15pt

\title[Bounds on depth of tensor products of modules]{Bounds on depth of tensor products of modules}

\author{Olgur Celikbas} \address{Department of Mathematics, 323
  Mathematical Sciences Bldg, University of Missouri--Columbia,
  Columbia, MO 65211 USA} \email{celikbaso@missouri.edu}

\author{Arash Sadeghi}
\address{School of Mathematics, Institute for Research in Fundamental Sciences (IPM), P.O. Box: 19395-5746, Tehran, Iran.}
\email{sadeghiarash61@gmail.com}

\author{Ryo Takahashi}
\address{Graduate School of Mathematics, Nagoya University, Furocho, Chikusaku, Nagoya 464-8602, Japan}
\email{takahashi@math.nagoya-u.ac.jp}
\urladdr{http://www.math.nagoya-u.ac.jp/~takahashi/}


\thanks{Sadeghi's research was partially supported by a grant from IPM (No. 92130026).}
\thanks{Takahashi was partly supported by JSPS Grant-in-Aid for Scientiﬁc Research (C) 25400038.}


\begin{abstract}
Let $R$ be a local complete intersection ring and let $M$ and $N$ be nonzero finitely generated $R$-modules. We employ Auslander's transpose in the study of the vanishing of Tor and obtain useful bounds for the depth of the tensor product $M\otimes_{R}N$. An application of our main argument shows that, if $M$ is locally free on the the punctured spectrum of $R$, then either $\depth(M\otimes_{R}N)\geq \depth(M)+\depth(N)-\depth(R)$, or  $\depth(M\otimes_{R}N)\leq \cod(R)$. Along the way we generalize an important theorem of  D. A. Jorgensen and determine the number of consecutive vanishing of $\Tor_i^R(M,N)$ required to ensure the vanishing of all higher $\Tor_i^R(M,N)$.
\end{abstract}

\maketitle{}

\setcounter{tocdepth}{1}

\section{Introduction}

This paper originates in an attempt to deal with the following question that was implicitly raised and studied by Huneke, Jorgensen and Wiegand:

\begin{ques} (\cite{HJW}, see also \cite[6.6]{CeRo}) \label{queintro} Let $R$ be a local complete intersection ring of codimension $c$ and let $M$ and $N$ be finitely generated $R$-modules. Assume $M\otimes_{R}N$ is a $(c+1)$st syzygy of some finitely generated $R$-module. Under what conditions is the pair $(M,N)$ Tor-independent, i.e., is $\Tor^{R}_{i}(M,N)=0$ for all $i\geq 1$?
\end{ques}

Recall that a local ring $(R, \mathfrak{m}, k)$ is said to be a \textit{complete intersection} if the $\mathfrak m$-adic completion $\widehat{R}$ of $R$ is of the form $Q/(\underline{f})$,
where $\underline{f}$ is a $Q$-regular sequence and $Q$ can be taken as a ring of formal power series over the field $k$, or over
a complete discrete valuation ring with residue field $k$. The nonnegative integer $\cod(R)=\edim(R)-\dim(R)$ is called the \textit{codimension} of $R$.

A remarkable consequence of Tor-independence over complete intersection rings $R$ is the \emph{depth formula}, $\depth(M)+\depth(N)=\depth(R)+\depth(M\otimes N)$, established by Auslander \cite[1.2]{Au} when $R$ is regular, and by Huneke and Wiegand \cite[2.5]{HW1} when $R$ is singular. The depth formula is central to homological commutative algebra and has been studied extensively; see for example \cite{ArY, BerJor, IC, CJ, I, Sa}. In particular, when $R$ is regular, it is a natural extension of the classical Auslander -- Buchsbaum formula: $\pd(M)+\depth(M)=\depth(R)$ \cite[3.7]{AusBuc}.

The condition that $M\otimes_{R}N$ has high depth properties is not enough for Tor-independence in general; see for example \cite[3.14]{Ce}. The motivation for Question \ref{queintro} comes from Auslander's seminal work: if $R$ is regular (i.e., $c=0$) and $M\otimes_{R}N$ is torsion-free, equivalently, is a first syzygy module, then $(M,N)$ is Tor-independent \cite[3.1]{Au} and \cite[Corollary 2]{Li}. Huneke and Wiegand proved that if $R$ is a hypersurface (i.e. $c=1$), $M\otimes_RN$ is reflexive (equivalently, is a second syzygy module) and either $M$ or $N$ has a rank, then $\Tor_i^R(M,N)=0$ for all $i\ge 1$ [32, 2.7]; see also \cite[1.4]{GO}. Huneke, Jorgensen and Wiegand \cite{HJW} analyzed the universal pushforward and quasi-liftings of modules, and obtained similar vanishing results for the cases where $c=2$ and $c=3$ with escalating assumptions; although their techniques break down when $c\geq 4$, their methods have proved noteworthy to examine torsion in tensor products of modules.

Question \ref{queintro} has been studied in \cite{Ce} for modules whose complexity is strictly less than the codimension of the complete intersection ring considered; see (\ref{cx}). By exploiting the vanishing of a generalized version of Hochster's $\theta(-,-)$ pairing, Dao \cite{Da2} obtained certain conditions on the modules $M$ and $N$ so that if $M\otimes_{R}N$ is a $(c+1)$st syzygy over a complete intersection -- in an unramified regular local ring -- of codimension $c$, then $\Tor^{R}_{i}(M,N)=0$ for all $i\geq 1$. More recently the unramified condition has been removed and Dao's result has been improved over smooth graded complete intersections \cite{GORS}; see also \cite{Mark1} and \cite{Walkernew}.

In this paper we consider Question \ref{queintro} for a pair of modules that satisfy the depth formula (\ref{df}) and develop techniques that are entirely different from those previously used in the literature; cf., \cite{Ce}, \cite{CeD2}, \cite{GORS}, \cite{Da2}, \cite{HJW}, \cite{HW1} and \cite{HW2}. In the following, $\len$ denotes length.

\begin{thm} \label{thmintro0} Let $R$ be a local complete intersection ring of codimension $c$ and let $M$ and $N$ be nonzero finitely generated $R$-modules. Assume:
\begin{enumerate}[\rm(i)]
\item $\len(\Tor^{R}_{i}(M,N))<\infty$ for all $i\geq 1$ (e.g., $M$ is locally free on the punctured spectrum of $R$.)
\item $\depth(M\otimes_{R}N)\geq c+1$ (e.g., $\dim(R)\geq c+1$ and $M\otimes_{R}N$ is a $(c+1)$st syzygy module.)
\end{enumerate}
Then $\depth(M\otimes_{R} N)\geq \depth(M)+\depth(N)-\depth(R)$. Moreover, $(M,N)$ is Tor-independent if and only if $(M,N)$ satisfies the depth formula.
\end{thm}

If a local ring $R$ has an \emph{isolated singularity}, i.e., $R_{p}$ is regular for all prime ideals $p$ in the punctured spectrum of $R$, and if $M\otimes_{R}N$ is a syzygy module, then it follows that $\len(\Tor^{R}_{i}(M,N))<\infty$ for all $i\geq 1$; see \cite[3.1]{Au} and \cite[Corollary 2]{Li}. Therefore, as a particular example, if $M\otimes_{R}N$ is a third syzygy module over the codimension two complete intersection isolated singularity $R=\CC[[x_1, \ldots, x_n]]/(x_1^2+x_2^2+\ldots+x_n^2,  x_1^2+2x_2^2+\dots+nx_n^2)$, with $n\geq 5$, then $\depth(M\otimes_{R}N)\geq 3$, and hence we conclude from Theorem \ref{thmintro0} that $(M,N)$ satisfies the depth formula (\ref{df})  if and only if $(M,N)$ is Tor-independent; see Question \ref{queintro}.

In addition to analyzing Question \ref{queintro}, Theorem \ref{thmintro0} yields useful bounds for depth of tensor products of modules; see Corollaries \ref{cor2} and \ref{cor21}. It determines a necessary condition for the depth formula and hence, from another point of view, it complements \cite[2.4]{CeD2}. Our proof of Theorem \ref{thmintro0} relies upon results of Auslander and Bridger \cite{AuBr}. A technical detail worth pointing out is that we assume $M\otimes_{R}N$ has sufficient depth, but we do not assume $M$, $N$ or $M\otimes_{R}N$ is a syzgy module, or equivalently, satisfies any Serre's conditions $(S_{n})$ \cite[3.8]{EG}, in Theorem \ref{thmintro0}; cf. \cite[3.4]{Ce}, \cite[2.2]{CeD2},  \cite[5.11]{GORS} and \cite[7.6]{Da2}.

Theorem \ref{thmintro0} becomes more interesting if we consider an application of it over hypersurface rings and apply a result of Huneke and Wiegand \cite[3.1]{HW1}. We establish in Corollary \ref{corthm0} that:

\begin{cor} \label{corthm1} Let $R$ be a hypersurface ring with an isolated singularity and let $M$ and $N$ be nonfree maximal Cohen-Macaulay $R$-modules. Then $\depth(M\otimes_{R}N)\leq 1$.
\end{cor}

A remarkable rigidity theorem of Jorgensen \cite[2.3]{Jo1} states that if $R$ is a $d$-dimensional complete intersection, $r=\min\{\cx(M), \cx(N)\}$ and $\Tor_{n}^R(M,N)=\cdots=\Tor_{n+r}^R(M,N)=0$ for some $n\geq d-b+1$, where $b=\max\{\depth(M), \depth(N)\}$, then $\Tor^{R}_{i}(M,N)=0$ for all $i\geq d-b+1$. In proving Theorem \ref{thmintro0}, we discover that our argument generalizes Jorgensen's theorem. Our rigidity result, stated as Proposition \ref{Torcor}, depends on the complexity $\cx(M,N)$ of the pair $(M,N)$ rather than the minimum of complexities, where $\cx(M,N)$ can be in general strictly less than $\min\{\cx(M), \cx(N)\}$; see Example \ref{eg6}.

\section{Definitions and Preliminary Results}

\begin{chunk} \textbf{Convention.} \label{why} Throughout the paper $R$ denotes a local ring, that is, a commutative Noetherian ring with unique maximal ideal $\fm$ and residue field $k$. All modules considered over $R$ are finitely generated. We have, by definition, $\depth(0)=\infty$ and $\pd(0)=-\infty$. For an $R$-module $X$, we set $X^{\ast}=\Hom_{R}(X,R)$.

\end{chunk}

\begin{chunk} \label{a1} \textbf{Auslander's Transpose.} (\cite{AuBr})
Let $M$ be an $R$-module with a projective presentation $P_1\overset{f}{\rightarrow}P_0\rightarrow M\rightarrow 0$. Then the \emph{transpose} $\Tr M$ of $M$
is the cokernel of $f^{\ast}=\Hom_{R}(f,R)$ and hence is given by the exact sequence:
$0\rightarrow M^*\rightarrow P_0^*\rightarrow P_1^*\rightarrow \Tr M\rightarrow 0$.

If $n$ is a positive integer, $\mathcal{T}_n M$ denotes the transpose of the $(n-1)st$ syzygy of $M$, i.e.,
\begin{equation} \tag{\ref{a1}.1}
\mathcal{T}_n M=\Tr\Omega^{n-1}M.
\end{equation}
There is an exact sequence of functors \cite[2.8]{AuBr}:
\begin{align}
\tag{\ref{a1}.2} \Ext^1_R(\mathcal{T}_{n+1}M,-) \hookrightarrow \Tor_n^R(M,-)  & \rightarrow  \Hom_R(\Ext^n_R(M,R),-)
 \rightarrow \Ext^2_R(\mathcal{T}_{n+1}M,-). \notag{}
\end{align}
\end{chunk}

\begin{chunk} \textbf{Gorenstein and complete intersection dimensions.} \label{CIdim} (\cite{AuBr, AGP, Larsbook})
A finitely generated $R$-module $M$ is said to be \emph{totally reflexive} if the
natural map $M \to M^{\ast\ast}$ is bijective and $\Ext_{R}^{i}(M,R)=0=\Ext_{R}^{i}(M^{\ast},R)$ for all $i\geq 1$.

The infimum of $n$ for which there exists an exact sequence
$0 \to X_{n} \to \dots  \to X_{0} \to M\to 0,$
such that each $X_{i}$ is totally reflexive, is called the \emph{Gorenstein dimension} of $M$. If $M$ has Gorenstein dimension $n$, we write $\G-dim(M)=n$. Therefore $M$ is totally reflexive if and only if $\G-dim(M)\leq 0$, where it follows by convention that $\G-dim(0)=-\infty$.

A diagram of local ring maps $R \to R' \twoheadleftarrow Q$ is called a {\em quasi-deformation} provided that $R \to R'$ is flat and the kernel of the surjection $R' \twoheadleftarrow Q$ is generated by a $Q$-regular sequence. The \emph{complete intersection dimension} of $M$ is:
\begin{equation*}
\CI(M) = \inf\{ \pd_Q(M \otimes_{R} R') - \pd_Q(R') \ | \
R \to R' \twoheadleftarrow Q \ {\text{is a quasi-deformation}}\}.
\end{equation*}

We catalogue a few key properties of complete intersection dimension:

\begin{enumerate}[\rm(i)]
\item If $R$ is a complete intersection ring, then $\CI(M)<\infty$; see \cite[1.3]{AGP}
\item If $\CI(M)<\infty$, then $\CI(M)=\G-dim(M)=\depth(R)-\depth(M)$; see \cite[1.4]{AGP}
\item If $\CI(M)=0$, then $\CI_{R_{p}}(M_{p})=0$ for all $p\in \Supp(M)$; see \cite[1.6]{AGP}
\item If $\CI(M)<\infty$, then $\Tor^R_{i}(M,N)=0$ for all $i\geq \CI(M)+1$ if and only if $\Tor^{R}_{i}(M,N)=0$ for all $i\gg 0$; see \cite[4.9]{AvBu}
\end{enumerate}
\end{chunk}

\begin{chunk} \textbf{Complexity.} \label{cx} (\cite{Av1}) The \textit{complexity} of a sequence of nonnegative integers $B =
\{b_i\}_{i\geq 0}$ is: $$\cx(B) = \inf\{r\in \NN \cup \{0\} \mid
b_n\leq A \cdot n^{r-1} \ \text {for some real number} \  A  \ \text
{and for all} \ n\gg 0 \}.$$ According to this notation, the
complexity of a pair $(M,N)$ of finitely generated $R$-modules can be
defined as \cite{AvBu}:
$$\cx(M,N) = \cx\left( \{ \rank_{k}(\Ext^{i}_{R}(M,N)\otimes_{R}k) \} \right)$$

The complexity $\cx(M)$ of $M$ is defined as $\cx(M,k)$ and it follows from the definition that $\cx(M,N)=\cx(\Omega^iM, N)$ for all nonnegative integers $i$. Moreover one has the following properties:

\begin{enumerate}[\rm(i)]
\item Assume $R$ is a complete intersection. Then,
\begin{enumerate}[(a)]
\item $\cx(M,N)\leq\text{min} \{ \cx(M), \cx(N) \} \leq \cod(R)$; see \cite[5.7]{AvBu}.
\item If $M$ is maximal Cohen-Macaulay, then $\cx(M^{\ast},N)=\cx(M,N)$; see \cite[5.6]{AvBu}.
\end{enumerate}

\item Assume $\CI(M)<\infty$. Then $\cx(M) \leq \edim(R)-\depth(R)$. If, in addition, $R$ is not a complete intersection, then the inequality is strict; see \cite[5.6]{AGP}.

\item If $\CI(M)=0$, then $\cx(M^*,N)\leq\cx(M^{\ast})=\cx(M)$; see \cite[4.1.2]{AvBu} and \cite[3.2]{BeJ3}.
\end{enumerate}
\end{chunk}

\begin{chunk} \textbf{Depth Formula.} \label{df} (\cite{Au})
Two finitely generated $R$-modules $M$ and $N$ satisfy the \emph{depth formula} provided that $\displaystyle{\depth(M) + \depth(N)= \depth(R) + \depth(M \otimes_R N)}$.

Huneke and Wiegand proved in \cite[2.5]{HW1} that Tor-independent modules over complete intersection rings satisfy the depth formula. The depth formula, for tensor products of finitely generated modules, is initially due to Auslander \cite{Au}; see also Christensen and Jorgensen \cite{CJ}, Foxby \cite{Foxby} and Iyengar \cite{I} for extensions of that formula to certain complexes of modules.
\end{chunk}

\section{Main result}

In this section we prove our main result, Theorem \ref{corthm1}: it subsumes Theorem \ref{thmintro0} advertised in the introduction. Section 4 contains several interesting applications of our result on depth of tensor products of modules.

\begin{thm}  \label{corthm1} Let $R$ be a local complete intersection ring and let $M$ and $N$ be nonzero finitely generated $R$-modules. Assume $\len(\Tor^{R}_{i}(M,N))<\infty$ for all $i\geq 1$. Assume further that $\depth(M\otimes_{R}N)\geq \cx(M,N)+1$. Then the following conditions are equivalent:
\begin{enumerate}[\rm(i)]

\item $\Tor^{R}_{i}(M,N)=0$ for all $i\geq 1$.
\vspace{0.03in}

\item $\depth(M)+\depth(N)\geq \depth(R)+\cx(M,N)$.
\vspace{0.03in}

\item $\depth(M)+\depth(N)=\depth(R)+\depth(M\otimes_{R}N)$.
\vspace{0.03in}

\item $\depth(M)+\depth(N)\geq \depth(R)+\depth(M\otimes_{R} N)$.
\vspace{0.03in}

\item $\depth(M)+\depth(N)\geq \depth(R)$ and $\Tor^R_{i}(M,N)=0$ for all $i\gg 0$.
\vspace{0.03in}

\item $\depth(M)+\depth(N)\geq \depth(R)$ and $\Ext_R^{i}(M,N)=0$ for all $i\gg 0$.
\end{enumerate}
\end{thm}

It requires substantial preparation to prove Theorem \ref{corthm1}. Our main argument is to prove (ii)$ \Longrightarrow $(i); see (\ref{df}). The equivalence of $(v)$ and $(vi)$, i.e., $\Tor^R_{i}(M,N)=0$ for all $i\gg 0$ if and only if $\Ext_R^{i}(M,N)=0$ for all $i\gg 0$, is due to Avramov and Buchweitz \cite[6.1]{AvBu}, and it holds independently of the depth inequality over arbitrary complete intersections. Also a result of Araya and Yoshino \cite[2.5]{ArY} gives (v)$  \Longrightarrow $(i) without any assumption on $\depth(M\otimes_{R}N)$; see Lemma \ref{L1}. Here we include conditions $(v)$ and $(vi)$ in Theorem \ref{corthm1} for completeness; note that we do not assume the vanishing of all higher $\Tor_{i}^{R}(M,N)$ to prove (ii)$ \Longrightarrow $(i); cf., \cite[2.7]{JAB}.

We start with a lemma which is crucial to our proof of Theorem \ref{corthm1}.  One can find a proof of Lemma \ref{pr1cor}(i, ii) in the unpublished manuscript of Sadeghi \cite[3.3 and 3.4]{Sa}; here a shorter argument is included for the convenince of the reader. We write $M\approx N$ to denote a \emph{stable isomorphism}, i.e., $M\oplus F\cong N\oplus G$ for some free modules $F$ and $G$.

\begin{lem}  \label{pr1cor}
Let $R$ be a local ring and let $N$ and $Y$ be nonzero finitely generated $R$-modules. Assume that $\CI(Y)=0$. Then,
\begin{enumerate}[\rm(i)]
\item $\CI(\Tr Y)=0$.
\item $\Ext^{i}_R(\Tr Y,N)=0$ for all $i\geq 1$ if and only if $\Tor_{i}^R(Y,N)=0$ for all $i\geq 1$.
\item $\CI(\mathcal{T}_iY)=0$ and $\mathcal{T}_iY\approx \Omega\mathcal{T}_{i+1}Y$ for all $i\geq 1$.
\item $\cx(\mathcal{T}_iY,N)=\cx(Y^{\ast},N)$ for all $i\geq 1$.
\item Let $r$ and $s$ be positive integers. Assume $\cx(Y^{\ast},N)\leq r-1$.  Assume further that $\Ext^{i}_{R}(\mathcal{T}_sY, N)=0$ for all $i=1, \ldots, r$. Then $\Tor^{R}_{i}(Y,N)=0$ for all $i\geq 1$.
\end{enumerate}
\end{lem}

\begin{proof}
We start by noting that $\CI(Y^*)=0$ \cite[3.5]{Be2}.
Moroever, by definition, we have that $Y^{\ast} \approx \Omega^2\Tr Y$. Thus it follows that $\CI(\Tr Y)<\infty$. As $Y$ is totally reflexive, so is $\Tr Y$ and hence $\CI(\Tr Y)=0$; see  \cite[4.1]{AuBr} and (\ref{CIdim})(ii). Consequently (i) follows.

We have, for all $i\in\mathbb{Z}$, that
$
\widehat{\Tor}_i^R(Y,N) \cong \widehat{\Ext}^{-i-1}_R(Y^{\ast}, N) \cong \widehat{\Ext}^{-i+1}_R(\Tr Y,N)
$; see for example \cite[4.4.7]{AvBu}.
Therefore,  \begin{align} \notag{}
\Tor_{i}^{R}(Y,N) =0 \text{ for all } i\geq 1
& \Longleftrightarrow  \widehat{\Tor}_i^R(Y,N)=0  \text{ for all } i \in \mathbb{Z} &\\ \notag{}
& \Longleftrightarrow  \widehat{\Ext}^{i}_R(\Tr Y,N)=0 \text{ for all } i \in \mathbb{Z}  \\  \notag{}
&  \Longleftrightarrow\Ext^{i}_R(\Tr Y,N)=0 \text{ for all } i\geq 1.
\end{align}
Here the first and the last equivalence follow from \cite[4.7 and 4.9]{AvBu}. This proves (ii).

Notice, since $Y$ is totally reflexive, $\Ext^{i}_{R}(Y,R)=0$ for all $i\geq 1$; see (\ref{CIdim}). Hence it follows from (\ref{a1}.1) that $\mathcal{T}_iY\approx\Omega\mathcal{T}_{i+1}Y$ for all $i\geq 1$. Moreover we have that $\CI(\Omega^{i-1}Y)=0$ for all $i\geq 1$; see \cite[1.9.1]{AGP}. Therefore, by (i), $\CI(\Tr \Omega^{i-1}Y)=0$  for all $i\geq 1$. As $\mathcal{T}_i Y=\Tr \Omega^{i-1}Y$ for all $i\geq 1$, (iii) follows.

Since $Y^{\ast} \approx \Omega^2\Tr Y$, it follows that $\cx(Y^{\ast}, N)=\cx(\Omega^2\Tr Y, N)=\cx(\Tr Y, N)=\cx(\mathcal{T}_{1}Y, N)$; see (\ref{cx}). Moreover $\cx(\mathcal{T}_{i+1} Y, N)=\cx(\Omega\mathcal{T}_{i+1} Y, N)$ which is, by (iii), equal to $\cx(\mathcal{T}_{i} Y, N)$ for all $i\geq 1$. This establishes (iv).

Now assume $r$ and $s$ are positive integers, $\cx(Y^{\ast},N)\leq r-1$ and that $\Ext^{i}_{R}(\mathcal{T}_sY, N)=0$ for all $i=1, \ldots, r$. Then it follows from (iv) that $\cx(\mathcal{T}_sY,N)\leq r-1$. Furthermore, by (iii), we see $\CI(\mathcal{T}_sY)=0$. Since $\Ext^{i}_{R}(\mathcal{T}_sY, N)=0$ for all $i=1, \ldots, r$, \cite[4.3]{CeD} implies that $\Ext^{i}_{R}(\mathcal{T}_sY, N)=0$ for all $i\geq 1$. Now we can make use of the stable isomorphism obtained in (iii) and deduce that $\Ext^{i}_R(\Tr Y,N)=0$ for all $i\geq 1$. Consequently (iv) follows from (ii).
\end{proof}

Auslander \cite[2.2]{Au} proved that finitely generated modules over unramified regular local rings are Tor-rigid, i.e., if $\Tor^{R}_{n}(M,N)=0$ for some nonnegative integer $n$, then $\Tor^{R}_{i}(M,N)=0$ for all $i\geq n$. Lichtenbaum \cite[Corollary 1]{Li} extended Auslander's result to arbitrary regular local rings. Murthy, in \cite[1.6]{Mu}, established a partial extension of Tor-rigidity over complete intersections: if $R$ has codimension $c$, and $\Tor^{R}_{n}(M,N)=\cdots=\Tor^{R}_{n+c}(M,N)=0$ for some nonnegative integer $n$, then $\Tor^{R}_{i}(M,N)=0$ for all $i\geq n$. It is well-known that Murthy's rigidity result --  and hence many of its consequences --  does not hold over Gorenstein rings (indeed over AB rings) that are not complete intersections; see \cite[2.14]{CeRo} and \cite[Theorem 2]{Sega}. Jorgensen \cite[2.3]{Jo1}, rather than considering vanishing intervals of lengths determined by the codimension of the ring, used the notion of complexity and studied the vanishing of Tor. We recall Jorgensen's result next; see also \cite[4.9]{AvBu}, \cite[3.6]{Be2} and \cite[2.3]{Jo3}.

\begin{thm} (Jorgensen \cite[2.3]{Jo1}) \label{Jorcx}
Let $R$ be a $d$-dimensional local complete intersection ring and let $M$ and $N$ be finitely generated $R$-modules. Set $r=\text{min} \{ \cx(M), \cx(N) \}$ and $b=\max\{\depth(M), \depth(N)\}$. If $\Tor^R_n(M,N)= \dots = \Tor^R_{n+r}(M,N)=0$
for some integer $n\geq d-b+1$, then $\Tor^{R}_{i}(M,N)=0$ for all $i\geq d-b+1$.
\end{thm}

We observe that the vanishing interval required in Theorem \ref{Jorcx} is determined by a finer bound, namely $\cx(M,N)$. Recall that, if $R$ is a complete intersection, then $\cx(X,N)=\cx(X^{\ast},N)$ for all maximal Cohen-Macaulay $R$-modules $X$; see (\ref{cx})(i)(b).

\begin{prop} \label{Torcor}
Let $R$ be a $d$-dimensional local ring and let $M$ and $N$ be finitely generated $R$-modules such that $\CI(M)<\infty$. Set $b=\depth(M)$ and assume $r$ is a positive integer with $\cx(M,N)\leq r-1$. Assume further that $\cx(X,N)=\cx(X^{\ast},N)$ for all finitely generated $R$-modules $X$ with $\CI(X)=0$ (e.g., $R$ is a complete intersection ring.) Then the following conditions are equivalent:
\begin{enumerate}[(i)]
\item $\Tor_n^R(M,N)=\ldots=\Tor_{n+r-1}^R(M,N)=0$ for some $n\geq d-b+1$.
\item $\Tor_i^R(M,N)=0$ for all $i\geq d-b+1$.
\end{enumerate}
\end{prop}

\begin{proof}
It suffices to prove (i) implies (ii). Assume (i) and set $L=\Omega^{n-1}M$. Since $n-1\geq d-b$, we have that $\CI(L)=0$; see \cite[1.9.1]{AGP}. Hence it follows from (\ref{a1}.2) and the stable isomorphism in Lemma \ref{pr1cor}(iii) that:
\begin{equation}\notag{}
\Ext^i_R(\mathcal{T}_{r+1}L,N) \cong \Ext^1_R(\mathcal{T}_{r+1-i+1}L,N) \hookrightarrow \Tor_{r+1-i}^R(L,N)=0 \text{ for all } i=1, \dots, r.
\end{equation}
We know $\cx(L^{\ast},N)=\cx(L,N)\leq r-1$. Therefore  we use Lemma \ref{pr1cor}(v) and conclude that $\Tor_i^R(L,N)=0$ for all $i\geq 1$. Now (ii) follows from Theorem \ref{Jorcx}.
\end{proof}

We have, over complete intersections, that $\cx(M,N)\leq
\text{min}\{\cx(M), \cx(N)\}$; see (\ref{cx})(i)(a). Thus
Proposition \ref{Torcor} is an extension of Theorem \ref{Jorcx}. We
give an example of finitely generated modules $M$ and $N$ such that
$0<\cx(M,N)<\min\{\cx(M), \cx(N)\}$: to not interrupt the flow we
defer it to the end of section 4; see Example \ref{eg6}. The
vanishing interval, determined by $\cx(M,N)$, in Proposition
\ref{Torcor} cannot be improved further, i.e., one cannot get by
with fewer consecutive vanishing Tors in general; see
\cite[3.11]{Ce} or \cite[4.1]{Jo1}.

Recently Christensen and Jorgensen \cite{CJ} established the \emph{derived depth formula} for certain complexes of modules, in particular for finitely generated modules, when all Tate Tors vanish. Hence the following is an application of Proposition \ref{Torcor}; see \cite[4.9]{AvBu} and \cite[5.2]{CJ} for details.

\begin{cor} Let $R$ be a $d$-dimensional local complete intersection ring and let $M$ and $N$ be finitely generated $R$-modules. Set $b=\max\{\depth(M),\depth(N)\}$ and assume $\cx(M,N)\leq r-1$ for some positive integer $r$. If $\Tor_n^R(M,N)=\ldots=\Tor_{n+r-1}^R(M,N)=0$ for some $n\geq d-b+1$, then Tate homology $\widehat {\Tor} _{i}^{R}(M,N)$ vanish for all $i\in \ZZ$ and the derived depth formula holds:$$\depth(M)+\depth(N)= \depth(R) +\depth(M \tensor_{R}  N)$$
\end{cor}

Next is an observation adopted from a result of Araya and Yoshino \cite[2.5]{ArY}: the idea  indeed goes back to Auslander \cite[1.2]{Au}; see also \cite[2.7]{JAB}.

\begin{lem} \label{L1} Let $R$ be a local ring and let $M$ and $N$ be finitely generated $R$-modules. Assume the following conditions hold:
\begin{enumerate}[\rm(i)]
\item $\CI(M)\leq \depth(N)$.
\item $\depth(\Tor^{R}_{i}(M,N)) \in \{0, \infty \}$ for all $i=1, \ldots, \CI(M)$.
\end{enumerate}
Then $\Tor_{i}^{R}(M,N)=0$ for all $i\gg 0$ if and only if $\Tor^{R}_{i}(M,N)=0$ for all $i\geq 1$.
\end{lem}

\begin{proof} Assume $\Tor_{i}^{R}(M,N)=0$ for all $i\gg 0$. If $\CI(M)=0$, then there is nothing to prove; see (\ref{CIdim})(iv). Therefore we may assume $\CI(M)\geq 1$.

Let $q$ be the largest nonnegative integer such that $\Tor^{R}_{q}(M,N) \neq 0$. There is nothing to prove if $q=0$. Hence suppose $q\geq 1$. Note that $\CI(M)\geq q \geq 1$; see (\ref{CIdim})(iv). Therefore, by  (ii), $\depth(\Tor^{R}_{q}(M,N))=0$. Now the depth formula of \cite[2.5]{ArY} gives: $$\depth(M)+\depth(N)=\depth(R)+ \depth(\Tor^{R}_{q}(M,N))-q=\depth(R)-q.$$
This implies $\depth(M)+\depth(N)< \depth(R)$ so that $\depth(N)<\CI(M)$; see (\ref{CIdim})(ii).
This contradicts (i). Therefore $q=0$ and hence $\Tor^{R}_{i}(M,N)=0$ for all $i\geq 1$.
\end{proof}

The following reduction argument detects a suitable module of complete intersection dimension zero for the proof of Theorem \ref{corthm1}.

\begin{lem}\label{le5}
Let $R$ be a local ring and let $M$ and $N$ be nonzero finitely generated $R$-modules. Assume $n$ is a nonnegative integer and the following conditions hold:
\begin{enumerate}[\rm(i)]
\item $\len(\Tor^{R}_{i}(M,N))<\infty$ for all $i\geq 1$.
\item $\CI(M)+n\leq \depth(N)$.
\item $\depth(M\otimes N)\geq n+1$.
\end{enumerate}
Then there exists a finitely generated $R$-module $L$ such that:
\begin{enumerate}[\rm(1)]
\item $\CI(L)=0$.
\item $\Tor^{R}_i(L,N)=0$ for all $i=1, \dots, n+1$.
\item $\Tor^{R}_{i+1+n}(L,N)\cong\Tor^{R}_i(M,N)$ for all $i\geq 1$.
\item $\cx(M,X)=\cx(L,X)$ for all finitely generated $R$-modules $X$.
\end{enumerate}
\end{lem}

\begin{proof}
We consider a \emph{finite projective hull} of $M$ \cite[3.1 and 3.3]{CLWIS},
i.e., a short exact sequence of finitely generated $R$-modules of
the form
\begin{equation} \tag{\ref{le5}.1}
0 \rightarrow M \rightarrow T \rightarrow Y \rightarrow 0,
\end{equation}
where $\pd(T)<\infty$ and $\G-dim(Y)=0$.

We shall first prove that the module $Y$ in (\ref{le5}.1) satisfies the following properties:
\begin{enumerate}[\rm(a)]
\item $\CI(Y)=0$.
\item $\Tor^{R}_1(Y,N)=0$.
\item $\depth(Y\otimes N)\geq n$.
\item $\Tor^{R}_{i+1}(Y,N)\cong\Tor^{R}_i(M,N)$ for all $i\geq 1$.
\item $\cx(M,X)=\cx(Y,X)$ for all finitely generated $R$-modules $X$.
\end{enumerate}

Notice, for all $R$-modules $X$, $\Ext_{R}^{i}(M, X) \cong \Ext^{i+1}_{R}(Y,X)$ for all $i \gg 0$. This justifies (e). Since
$\CI(M)<\infty$ and $\pd(T)<\infty$, it follows from (\ref{le5}.1) and \cite[3.6]{Sean} that $\CI(Y)<\infty$. Therefore
$\CI(Y)=\G-dim(Y)=0$; see (\ref{CIdim})(ii). Thus (a) follows.

Dualizing (\ref{le5}.1) with respect to $R$ and using the fact that
$\Ext_{R}^{i}(Y,R)=0$ for all $i\geq 1$ (recall that $Y$ is totally reflexive), we conclude
that $\Ext_{R}^{i}(T,R) \cong \Ext_{R}^{i}(M,R)$ for all $i\geq 1$.
This implies $\pd(T)=\G-dim(M)$ \cite[\S3.2.2, Corollaire]{Tome1} and hence $\depth(T)=\depth(M)$; see
(\ref{CIdim})(ii). Therefore, by (ii), we obtain:
\begin{equation} \tag{\ref{le5}.2}
\CI(T)=\depth(R)-\depth(T)=\depth(R)-\depth(M)=\CI(M)\leq \depth(N).
\end{equation}
It follows from (\ref{le5}.1) that $\Tor^{R}_{i+1}(Y,N)\cong
\Tor^{R}_{i}(M,N)$ for all $i\gg 0$. Hence, by (i),
$\len(\Tor^{R}_{i}(Y,N))<\infty$ for all $i\gg 0$. Since
$\CI_{R_{p}}(Y_{p})=0$, it follows that $\Tor^{R}_{i}(Y,N)_{p}=0$
for all $i\geq 1$ and for all  $p\in \Spec(R)-\{\mathfrak{m}\}$; see
(\ref{CIdim})(iii, iv). Therefore,
\begin{equation} \tag{\ref{le5}.3}
\len(\Tor^{R}_{i}(Y,N))<\infty \text{ for all } i\geq 1.
\end{equation}
Now, using (i) and (\ref{le5}.3), we deduce from the short exact sequence in (\ref{le5}.1) that:
\begin{equation} \tag{\ref{le5}.4}
\len(\Tor^{R}_{i}(T,N))<\infty \text{ for all } i\geq 1.
\end{equation}
Therefore, in view of (\ref{le5}.2) and (\ref{le5}.4), Lemma
\ref{L1} implies that $\Tor^{R}_{i}(T,N)=0$ for all $i\geq 1$ (recall that $\pd(T)<\infty$.)
Consequently (d) follows from (\ref{le5}.1). Moreover the depth formula (\ref{df}) holds for the pair $(T,N)$. Therefore, since $\depth(T)=\depth(M)$, we have by (ii) that:
\begin{equation} \tag{\ref{le5}.5}
\depth(T\otimes_{R}N)=\depth(N)-(\depth(R)-\depth(M))=\depth(N)-\CI(M)\geq n.
\end{equation}

Now we apply $N\otimes_{R}-$ to (\ref{le5}.1) and obtain the exact sequence:
\begin{equation} \tag{\ref{le5}.6}
0\rightarrow \Tor^{R}_{1}(Y,N) \rightarrow M\otimes_{R}N \rightarrow T\otimes_{R}N \rightarrow Y\otimes_{R}N \rightarrow 0.
\end{equation}
As $\len(\Tor^{R}_{1}(Y,N))<\infty$ and $\depth(M\otimes_{R}N)\geq 1$, it follows that $\Tor^{R}_{1}(Y,N)=0$; this yields (b).

Finally the depth lemma, applied to the short exact sequence in (\ref{le5}.6), yields
$$\depth(Y\otimes_{R}N) \geq \min \{\depth(T\otimes_{R}N), \depth(M\otimes_{R}N)-1 \}.$$
Thus (c) follows from (iii) and (\ref{le5}.5). This establishes the properties of $Y$ stated in (a) -- (e).

We can now proceed to construct the module $L$ as claimed. If $n=0$, it is enough to choose $L=Y$. So we assume $n\geq 1$. Since $Y$ is totally reflexive, it follows from \cite[Chapitre 3, Proposition 8]{Tome1} that there
exists a short exact sequence of finitely generated $R$-modules
\begin{equation}\tag{\ref{le5}.7}
0\to Y \to R^{(t)} \to Y_{1} \to 0,
\end{equation}
where $\CI(Y_{1})=0$; see (\ref{CIdim})(ii) and \cite[3.6]{Sean}.
Since $\Tor^R_{i+1}(Y_1, N) \cong \Tor_{i}^{R}(Y,N)$ for
all $i\geq 1$, we use (\ref{CIdim})(iii, iv) and deduce from (\ref{le5}.3)
that $\len(\Tor^{R}_{i}(Y_1,N))<\infty$ for all $i\geq 1$. Recall
that $\depth(Y\otimes_{R}N)\geq n\geq 1$; see (c). Thus, tensoring (\ref{le5}.7)
with $N$, we see  $\Tor_{1}^R(Y_1,N)=0$ and obtain the short exact sequence
\begin{equation}\tag{\ref{le5}.8}
0\to Y\otimes_{R}N \to N^{(t)} \to Y_{1}\otimes_{R}N \to 0.
\end{equation}
Since $\depth(N)\geq n$, the
depth lemma applied to (\ref{le5}.8) shows that
$\depth(Y_1\otimes_RN)\geq n-1$; see (ii). Recall $\Tor_1^R(Y_1, N)=0$
and that $\Tor_2^R(Y_1, N)\cong \Tor_1^R(Y,N)=0$; see (b). Moreover $\Tor_{i+2}^R(Y_1,N)\cong\Tor_i^R(M,N)$ for all
$i\geq 1$ and $\cx(M,X)=\cx(Y,X)=\cx(Y_1, X)$ for all $R$-modules $X$.

Now, if $n=1$, pick $L=Y_1$ and we are done. If not, since $\depth(Y_1\otimes_{R}N)\geq n-1$, we proceed similarly and repeat the previous argument $n-1$ times. More precisely, for each $j=1,\dots, n$,
we obtain a finitely generated $R$-module $Y_{j}$ such that
$\CI(Y_{j})=0$, $\Tor_i^R(Y_j, N)=0$ for all $i=1, \ldots, j+1$ and
$\Tor^R_{i+j+1}(Y_j, N) \cong \Tor_{i+1}^{R}(Y,N)\cong\Tor_i^R(M,N)$
for all $i\geq 1$. In particular, $\Tor_i^R(Y_n, N)=0$ for all $i=1, \ldots,
n+1$ and $\Tor_{l+n+1}^R(Y_n, N)\cong\Tor_l^R(M,N)$ for all
$l\geq 1$. As $\cx(M,X)=\cx(Y,X)=\cx(Y_n,X)$ for all
$R$-modules $X$, setting $L=Y_n$ completes the proof.
\end{proof}

Our next result, Theorem \ref{t3}, is technical in nature, but
it turns out to be quite useful for applications; see section 4.
Furthermore its hypotheses are easier to comprehend when $R$ is a complete
intersection ring: for example the complexity equality in (b) is
satisfied; see (\ref{cx})(i)(b). Similarly if the pair $(M,N)$ satisfies the depth formula (\ref{df}) and $\depth(M\otimes_{R}N)\geq n+1$, then
the conditions in (ii) hold; see (\ref{CIdim})(ii).

\begin{thm}\label{t3}
Let $R$ be a local ring and let $M$ and $N$ be nonzero finitely generated $R$-modules. Assume $n$ is a nonnegative integer and
\begin{enumerate}[\rm(i)]
\item $\len(\Tor^{R}_{i}(M,N))<\infty$ for all $i\geq 1$.
\item $\depth(N) \geq \CI(M)+n$ and $\depth(M\otimes_{R}N)\geq n+1$.
\end{enumerate}
Assume further that  at least one of the following conditions holds:
\begin{enumerate}[\rm(a)]
\item $n\leq \cx(M)$ and $\Tor^{R}_{i}(M,N)=0$ for all $i=1, \dots, \cx(M)-n$.
\item $\max\{n,\cx(M,N)\}\leq r-1$ for some integer $r$,  $\Tor^{R}_{i}(M,N)=0$ for all $i=1, \dots, r-n-1$, and $\cx(X,N)=\cx(X^{\ast},N)$ for all finitely generated $R$-modules $X$ with $\CI(X)=0$.
\end{enumerate}
Then $\Tor^{R}_{i}(M,N)=0$ for all $i\geq 1$.
\end{thm}

\begin{rmk} If $r=n+1$, we do not assume $\Tor^{R}_{i}(M,N)$ vanishes in Theorem \ref{t3}(b).
\end{rmk}

\begin{proof}[Proof of Theorem \ref{t3}] Set $s=\cx(M)$. If $s=0$, i.e.,
if $\pd(M)<\infty$, then in view of the hypotheses (i) and (ii), Lemma \ref{L1} implies that $\Tor^{R}_{i}(M,N)=0$ for all $i\geq 1$. Hence we may assume $s\geq 1$ through the rest of the proof.

It follows from Lemma \ref{le5} that there exits a finitely generated $R$-module $L$ such that $\CI(L)=0$ and the following conditions hold:
\begin{enumerate}[(1)]
\item $\cx(M,X)=\cx(L,X)$ for all finitely generated $R$-modules $X$.
\item $\Tor^{R}_i(L,N)=0$ for all $i=1, \dots, n+1$.
\item $\Tor^{R}_{i+n+1}(L,N)\cong \Tor^{R}_{i}(M,N)$ for all $i\geq 1$.
\end{enumerate}
Therefore, by (3), it suffices to prove that $\Tor^{R}_{i}(L,N)=0$ for all $i\geq 1$.

Assume (a). Then $n\leq s$ and $\Tor^{R}_{i}(M,N)=0$ for all $i=1, \dots, s-n$. Now we use the isomorphism in (3) and deduce that $\Tor^{R}_{i}(L,N)=0$ for all $i=n+2, \dots, s+1$. Thus, by (2), we see $\Tor^{R}_{i}(L,N)=0$ for all $i=1, \dots, s+1$. Since $\CI(L)=0$ and $s=\cx(M)=\cx(L)$, Jorgensen's result \cite[2.3]{Jo3} yields the vanishing of $\Tor^{R}_{i}(L,N)$ for all $i\geq 1$.

Next assume (b). Notice, by (1), we have that $\cx(L,N)=\cx(M,N)\leq r-1$. Therefore, using Proposition \ref{Torcor} for the pair $(L,N)$, it suffices to see that $\Tor^{R}_i(L,N)=0$ for all $i=1, \dots, r$. Recall that $r\geq n+1$. If $r=n+1$, then (2) gives the desired result. If, on the other hand, $r>n+1$, we obtain from (b) and (3) that $\Tor^{R}_{i+n+1}(L,N)\cong \Tor^{R}_{i}(M,N)=0$ for all $i=1, \ldots, r-n-1$, i.e., $\Tor^{R}_{i}(L,N)=0$ for all $i=n+2, \ldots, r$. Hence (2) yields that $\Tor^{R}_i(L,N)=0$ for all $i=1, \dots, r$.
\end{proof}

We are now ready to prove Theorem \ref{corthm1}. Recall that Theorem \ref{corthm1} subsumes Theorem \ref{thmintro0} which is stated in the introduction.

\begin{proof}[Proof of Theorem \ref{corthm1}] The implications $(i) \Longrightarrow (iii)$ $\Longrightarrow (iv) \Longrightarrow (ii)$ are clear; see (\ref{df}). The fact $(vi)  \Longleftrightarrow (v)$, i.e., $\cx(M,N)=0 \Longleftrightarrow \Tor_i^R(M, N)=0$ for all $i\gg 0$ is due to \cite[6.1]{AvBu}; see also (\ref{cx}). So we have that $(vi) \Longleftrightarrow (v) \Longrightarrow (ii)$. Since $(i) \Longrightarrow (v)$, it is enough to prove $(ii) \Longrightarrow (i)$.

Assume $(ii)$. Then, setting $n=\cx(M,N)$ and $r=\cx(M,N)+1$, we see that the hypotheses of Theorem \ref{t3} (with part (b)) hold; see (\ref{CIdim})(ii) and (\ref{cx})(i)(b). This implies that $\Tor_i^R(M, N)=0$ for all $i\geq 1$, i.e., $(i)$ follows.
\end{proof}

Our arguments raise some questions which we pose for future study:

\begin{ques} \label{q1} Does the conclusion of Theorem \ref{corthm1} hold if one replaces the hypothesis ``$\len(\Tor_{i}^{R}(M,N))<\infty$ for all $i\geq 1$" with ``$\len(\Tor_{i}^{R}(M,N))<\infty$ for all $i\gg 0$"?
\end{ques}

\begin{ques} \label{q2} Let $R$ be a local ring and let $M$ and $N$ be finitely generated $R$-modules. Assume $\CI(M)=0$. Is $\cx(M^{\ast},N)=\cx(M,N)$?
\end{ques}

An affirmative answer to Question \ref{q1} implies that the conclusion of Theorem \ref{corthm1} holds -- without the assumption that $\len(\Tor_{i}^{R}(M,N))<\infty$ for all $i\geq 1$ -- over complete intersection rings that are isolated singularities; see also Example \ref{eg0}. On the other hand, an affirmative answer to Question \ref{q2} allows one to replace the hypothesis ``$R$ is a complete intersection" with ``$M$ has finite complete intersection dimension" in Theorem \ref{corthm1}.

\section{Applications and Examples}

In this section we give several applications and examples. We start by recording the following reformulation of  Theorem \ref{corthm1}: it underlines the useful bounds we obtain for depth of tensor products of modules over complete intersections.

\begin{cor} \label{cor2} Let $R$ be a local complete intersection ring and let $M$ and $N$ be nonzero finitely generated $R$-modules. Assume $\len(\Tor^{R}_{i}(M,N))<\infty$ for all $i\geq 1$. Then
at least one of the following conditions hold:
\begin{enumerate}[\rm(i)]
\item $\depth(M\otimes_{R} N)\geq \depth(M)+\depth(N)-\depth(R)$.
\item $\depth(M\otimes_{R} N)\leq \cx(M,N)\leq \min\{ \cx(M), \cx(N) \} \leq \cod(R)$.
\end{enumerate}
\end{cor}

We note a consequence of Corollary \ref{cor2} and illustrate how to use Theorem \ref{corthm1} to study torsion in tensor products of modules. Recall that if a local complete intersection ring $R$ has an isolated singularity and $M\otimes_{R}N$ is a first syzygy module, equivalently is torsion-free \cite[3.5]{EG}, then $\len(\Tor^{R}_{i}(M,N))<\infty$ for all $i\geq 1$; see the discussion following Theorem \ref{thmintro0}.

\begin{cor} \label{cortorsion} Let $R$ be a local complete intersection ring with an isolated singularity and let $M$ and $N$ be nonzero finitely generated $R$-modules. Assume that the following holds: $$\depth(M)+\depth(N)-\depth(R)>\depth(M\otimes_{R} N)>\cx(M,N).$$
Then $M\otimes_{R}N$ has torsion.
\end{cor}

It is well-known that tensor products of nonzero modules generally have torsion. Hence the conclusion of Corollary \ref{cortorsion} might seem trivial. However, somewhat suprisingly, the assumption that $\depth(M\otimes_{R} N)>\cx(M,N)$ cannot be dropped in general.

\begin{eg} (\cite[4.1]{HW1}) \label{HWexample} Let $k$ be a field and put $R=k[[x,y,u,v]]/(uv-xy)$. Let $I=(x,u)$ and $L=(y,u)$ be the ideals of $R$. Then $R$ is a hypersurface with an isolated singularity and $I$ and $L$ are maximal Cohen-Macaulay $R$-modules. Moreover $\depth(I\otimes_{R}L)=1=\cx(I,L)$ and $I\otimes_{R}L$ is torsion-free.
\end{eg}

Next we point out that the conclusion of Corollary \ref{cor2} does not hold in general unless $\Tor^{R}_{i}(M,N)$ has finite length for all $i\geq 1$; see also Question \ref{q1}.

\begin{eg} \label{eg0} Let $k$ be a field, $R=k[[x,y,z,u]]/(xy)$, $M=R/(x)$ and $N=R/(y)$. Then $R$ is a three-dimensional hypersurface and $\depth(M)=\depth(N)=3$. Furthermore, for all $i\geq 0$, we have $\Tor_{2i}^R(M,N) \cong R/(x,y) \cong k[[z,u]]$ and that $\Tor_{2i+1}^R(M,N)=0$. Therefore $\len(\Tor_{2i}^R(M,N))=\infty$ for all $i\geq 0$. Note that neither of the claimed inequalities of Corollary \ref{cor2} hold:
$\cx(M,N)=1<2=\depth(M\otimes_{R}N)< \depth(M)+\depth(N)-\depth(R)=3$.
\end{eg}

The finite length hypothesis is essential for Theorem \ref{t3}: this can be seen from Example \ref{eg0} by setting $n=0$ or $n=1$. Example \ref{eg1} is concerned with the case where $R$ is not a hypersurface.

\begin{eg} \label{eg1} Let $k$ be a field, $R=k[[x,y,z,u]]/(xy,zu)$, $M=R/(x)$ and $N=R/(y^2+u^2)$. Then $R$ is a complete intersection of codimension and dimension two. One can see $\cx(M)=1$ and that $\Tor^{R}_1(M,N)=0$. Furthermore $\len(\Tor^{R}_2(M,N))=\infty$ since
\begin{align}\notag{}
\Tor^R_2(M,N)\cong (0:_Ny)/xN \cong (x,yz)N/xN & \cong (x,yz,y^2+u^2)/(x,y^2+u^2) &\\
& \cong R/((x,y^2+u^2):yz)  = R/(x,y,u) \cong k[[z]] \notag{}
\end{align}
Finally note $\CI(M)=\depth(R)-\depth(M)=0$ and that $\depth(M\otimes_{R}N)=1$. Now set $n=0$ and consider the hypotheses of Theorem \ref{t3}: (ii) and (a) hold but (i) fails.
\end{eg}

We are able to  improve Corollary \ref{cor2} in case $M$ is maximal Cohen-Macaulay.

\begin{cor} \label{cor22} Let $R$ be a local complete intersection ring and let $M$ and $N$ be nonzero finitely generated $R$-modules. Assume:
\begin{enumerate}[\rm(i)]
\item $\len(\Tor^{R}_{i}(M,N))<\infty$ for all $i\gg 0$ (e.g. $R$ has an isolated singularity.)
\item $M$ is maximal Cohen-Macaulay.
\item $\depth(N)\neq \depth(M\otimes_{R} N)$.
\end{enumerate}
Then  $\min\{\depth(N), \depth(M\otimes_RN)\}\leq \cx(M,N) \leq \min\{\cx(M), \cx(N)\} \leq \cod(R)$.
\end{cor}

\begin{proof} Let $r=\min\{\depth(N), \depth(M\otimes_RN)\}$. It suffices to prove $r\leq \cx(M,N)$; see (\ref{cx})(i)(a). Notice $\len(\Tor^{R}_{i}(M,N))<\infty$ for all $i\geq 1$; see (\ref{CIdim})(iii, iv). Suppose $\cx(M,N)\leq r-1$. Then, setting $n=r-1$, we deduce from Theorem \ref{t3}(b) that $\Tor_{i}^R(M,N)=0$ for all $i\geq 1$; see (\ref{cx})(i)(b). Hence the depth formula (\ref{df}) holds and contradicts (iii). Thus $r\leq \cx(M,N)$.
\end{proof}

Our next application yields an interesting result for hypersurface rings. It is an immediate consequence of Corollary \ref{cor22}, though it can also be proved by using the second rigidity theorem of Huneke and Wiegand \cite[2.7]{HW1}; see also Example \ref{eg0}.

\begin{cor} \label{corthm0} Let $R$ be a hypersurface  with an isolated singularity and let $M$ and $N$ be nonfree maximal Cohen-Macaulay $R$-modules. Then $\depth(M\otimes_{R}N)\leq 1$.
\end{cor}

\begin{proof} Set $d=\dim(R)$. We may assume $d\geq 2$. Then, since it has an isolated singularity, $R$ is normal and hence is a domain. Consequently, if $M\otimes_{R}N$ is maximal Cohen-Macaulay, i.e., if $\depth(M\otimes_{R}N)=\depth(N)$, then \cite[3.1]{HW1} shows that $M$ or $N$ is free. Therefore $\depth(M\otimes_{R}N) < \depth(N)$ and hence Corollary \ref{cor22} implies that $\depth(M\otimes_RN)\leq 1$.
\end{proof}

We briefly discuss Tor-rigidity: it is a necessary condition for the depth of $M\otimes_{R}N$ to be zero in Corollary \ref{corthm0}. Recall that $M$ is called Tor-rigid if each $R$-module $N$ has the property that $\Tor_{n}^{R}(M,N)=0$ for some nonnegative integer $n$ implies that $\Tor_{i}^{R}(M,N)=0$ for all $i\geq n$.

\begin{obs} \label{cons1} Let $R$ be a hypersurface  with an isolated singularity and let $M$ and $N$ be nonfree maximal Cohen-Macaulay $R$-modules. If $M$ is Tor-rigid, then $\depth(M\otimes_{R}N)=0$.
\end{obs}

\begin{proof} Assume $M$ is Tor-rigid and set $d=\dim(R)$. There is nothing to prove if $d=0$. Thus we assume $d\geq 1$. Since $R$ is Gorenstein and $N$ is maximal Cohen-Macaulay, there is a short exact sequence $0 \to N \to R^{(n)} \to C\to 0$ of finitely generated $R$-modules. Suppose now $\depth(M\otimes_RN)\geq 1$, i.e., $M\otimes_{R}N$ is torsion-free. Notice $\Tor_1^{R}(M, C)$ is torsion since $R$ is reduced. Hence, tensoring the above exact sequence  with $M$, we obtain $\Tor_{1}^{R}(M, C)=0$. It now follows from the Tor-rigidity assumption that $\Tor_i^{R}(M, C)=0$ for all $i\geq 1$; this forces $\pd(M)<\infty$ or $\pd(C)<\infty$ \cite[1.9]{HW2}, i.e., either $M$ or $N$ is free. Therefore $\depth(M\otimes_{R}N)=0$.
\end{proof}

In general, Tor-rigidity is a subtle condition to detect, but here is a concrete consequence of our observation in \ref{cons1}; see \cite[2.8 and 3.16]{Da1} for details.

\begin{cor} \label{corsimpsing} Let $R=\CC[[x_0, \dots, x_d]]/(f)$ be a hypersurface ring, where $d$ is a positive even integer and $0\neq f \in (x_0, \dots, x_d)^2$. If $R$ is a simple singularity, then $\depth(M\otimes_{R}N)=0$ for all nonfree maximal Cohen-Macaulay $R$-modules $M$ and $N$.
\end{cor}

Next we focus on analyzing depths of $M\otimes_{R}M$ and $M\otimes_{R}M^{\ast}$, where  $M^{\ast}$ denotes $\Hom_{R}(M,R)$. The following remark will be useful.

\begin{rmk} (\cite[Theorem III and 4.2]{AvBu} and \cite[2.1]{HJ}) \label{Torpd} Let $R$ be a local complete intersection ring and let $M$ be a finitely generated $R$-module.
\begin{enumerate}[\rm(i)]
\item If $\Tor_{i}^R(M,M)=0$ for all $i\gg 0$, then $\pd(M)<\infty$.
\item If $M$ is maximal Cohen-Macaulay and $\Tor_{i}^R(M, M^{\ast})=0$ for all $i\gg 0$, then $M$ is free.
\end{enumerate}
\end{rmk}

\begin{cor} \label{cor21} Let $R$ be a local complete intersection ring and let $M$ be a nonfree maximal Cohen-Macaulay $R$-module that is locally free on the punctured spectrum of $R$.
Then,
$$\max\{\depth(M\otimes_{R}M), \depth(M\otimes_{R}M^{\ast}) \}\leq \cx(M)=\cx(M^{\ast})\leq \cod(R).$$
\end{cor}

\begin{proof} Assume $\depth(M\otimes_{R} M)\geq \cx(M)+1$. Then it follows from Theorem \ref{corthm1} that $\Tor_{i}^{R}(M,M)=0$ for all $i\geq 1$. This implies that $\pd(M)<\infty$, i.e., $M$ is free; see Remark \ref{Torpd}(i). Therefore $\depth(M\otimes_{R} M)\leq \cx(M)$.

We now proceed to prove that $\depth(M\otimes_{R}M^{\ast})\leq \cx(M^{\ast})$; see (\ref{cx})(iii). Notice $M^{\ast}$ is a nonfree maximal Cohen-Macaulay $R$-module. If $M\otimes_{R}M^{\ast}$ is not maximal Cohen-Macaulay, then $\depth(M\otimes_{R}M^{\ast}) \neq \depth(M^{\ast})$ so that Corollary \ref{cor22} gives the required conclusion. Hence we may assume $M\otimes_{R}M^{\ast}$ is maximal Cohen-Macaulay. If $\depth(R) \geq \cx(M^{\ast})+1$, then Theorem \ref{corthm1} implies that $\Tor_{i}^R(M,M^{\ast})=0$ for all $i\geq 1$, which forces $M$ to be free; see Remark \ref{Torpd}(ii). Thus we have that $\depth(R) \leq \cx(M^{\ast})$. Since $\depth(M\otimes_{R}M^{\ast})=\depth(R)$, we are done.
\end{proof}

Assume $R$ and $M$ are as in Corollary \ref{cor21}. If $\dim(R)$ is even, then the depth of $M\otimes_{R}M^{\ast}$ is well understood, i.e.,  it follows that $\depth(M\otimes_{R}M^{\ast}) =0$; see \cite[3.10]{CeD2}. On the other hand, if $\dim(R)$ is odd, $M\otimes_{R}M^{\ast}$ may have positive depth; see for example \cite[3.12]{CeD2}.

The rest of the paper is devoted to providing two examples that emphasize the sharpness of our results. First we record a special case of Theorem \ref{t3}(a), the case where $n=0$ and $r=\edim(R)-\depth(R)$; see (\ref{cx})(ii) and cf. \cite[7.6]{Da2}.

\begin{cor} \label{corc4}
Let $R$ be a local ring and let $M$ and $N$ be nonzero finitely generated $R$-modules. Set $r=\edim(R)-\depth(R)$ and assume the following conditions hold:
\begin{enumerate}[(i)]
\item $\CI(M)<\infty$.
\item $\len(\Tor^{R}_{i}(M,N))<\infty$ for all $i\geq 1$.
\item $\Tor^{R}_{i}(M,N)=0$ for all $i=1, \dots, r$.
\item $\depth(M)+\depth(N)-\depth(R) \geq \depth(M\otimes_{R}N) \geq 1$.
\end{enumerate}
Then $\Tor^{R}_{i}(M,N)=0$ for all $i\geq 1$.
\end{cor}

We show by Example \ref{eg5} that finite complete intersection dimension hypothesis of Corollary \ref{corc4} cannot be removed in general. The reason why we set $r= \edim(R)-\depth(R)$ and refer to Corollary \ref{corc4} instead of Theorem \ref{t3} is because we do not know an example of  finitely generated modules $M$ and $N$ such that $\depth(M)+\depth(N)-\depth(R) \geq \depth(M\otimes_{R}N) \geq 1$, $\cx(M)<\infty$ and  $\CI(M)=\infty$. Note that $\cx(M)=\infty$ in Example \ref{eg5}.

\begin{eg} \label{eg5} Let $k$ be a field, $R=k[[x,y,z]]/(xy,yz)$, $M=R/(x)$ and $N=R/(z)$.
Then $\depth(M)+\depth(N)-\depth(R)=1+1-1=1=\depth(M\otimes_{R}N)$. Moreover
$\dim(R)=2$, $\depth(R)=1$ and $r=\edim(R)-\depth(R)=3-1=2$.

Let $p\in \Spec(R)$. Then $p$ contains $x$ or $y$ and hence $N_{p}$ is zero. This shows that $\len(\Tor^{R}_{i}(M,N))<\infty$ for all $i\geq 1$. We shall prove  $\Tor^R_1(M,N)=\Tor^R_2(M,N)=0\neq \Tor_3^R(M,N)$ and that $\CI(M)=\infty$.

The minimal free resolution $F_{\bullet}$ of $M$ is as follows:
$$\xymatrix{\ldots \ar[r] & R^{(3)} \ar[rrr]^-{\left[\begin{array}{ccc} y & -x & 0 \\
0 & z & y \\ \end{array}\right]}& &
& R^{(2)}   \ar[rr]^-{\left[\begin{array}{cc} z & x\end{array}\right]}&
& R^{}   \ar[r]^{y}
& R \ar[r]^{x} & R \ar[r]
&0}$$
Suppose $\CI_{R}(M)<\infty$. Notice $R$ is not a complete intersection. Hence we have that $\cx(M)<\edim(R)-\depth(R)=2$, i.e., $\cx(M)=1$; see (\ref{cx})(ii). This forces $F_{\bullet}$ to be periodic of period at most two after 1 step \cite[7.3(1)]{AGP}, which is incorrect. Thus $\CI(M)=\infty$.

One can see $(0:_{R}x)=(y)=(0:_{R}z)$ and that $(0:_{R}y)=(x,z)$. Set $L=R/(y)$. Then there are short exact sequences:
\begin{align} \tag{\ref{eg5}.1}
 0 \to L \stackrel{x}{\longrightarrow} R \to M \to 0
\end{align}
\begin{equation} \tag{\ref{eg5}.2}
0 \to L \stackrel{z}{\longrightarrow} R \to N \to 0, \text{ and }
\end{equation}
\begin{equation} \tag{\ref{eg5}.3}
0 \to R/(x,z) \stackrel{y}{\longrightarrow} R  \to L \to 0.
\end{equation}

Applying $-\otimes_{R}N$ to (\ref{eg5}.1) gives the exact sequence
$0 \to \Tor^R_1(M,N) \to R/(y,z) \stackrel{x}{\longrightarrow}  R/(z)$. Since $R/(y,z)\cong k[[x]]$ and $R/(z)\cong k[[x,y]]/(xy)$,
we conclude that $\Tor^R_1(M,N)=0$.

Applying $-\otimes_{R}L$ to (\ref{eg5}.2) gives the exact sequence
$0 \to \Tor^R_1(L,N) \to L \stackrel{z}{\longrightarrow} L$. Since $L \cong k[[x,z]]$, it follows that $\Tor^R_1(L,N)=0$. Notice $\Tor^R_2(M,N)\cong \Tor^R_1(L,N)$; see (\ref{eg5}.1). Therefore $\Tor^R_2(M,N)=0$.

Finally applying $-\otimes_{R}L$ to (\ref{eg5}.3) gives the exact sequence
$0 \to \Tor^R_1(L,L) \to k \stackrel{y}{\longrightarrow} L$. This implies that $\Tor^R_1(L,L)\cong k$. Notice $\Tor^R_3(M,N)\cong \Tor^R_2(L,N) \cong
\Tor^R_1(L,L)$; see (\ref{eg5}.1) and (\ref{eg5}.2). Thus $\Tor^R_3(M,N)\neq 0$.

Consequently the hypotheses (ii), (iii) and (iv) of Corollary \ref{corc4} hold but (i) fails for  $(M,N)$.
\end{eg}

Our next example is concerned with a pair of modules $M$ and $N$ over a complete intersection ring such that $1\leq \cx(M,N)<\text{min}\{\cx(M),\cx(N) \}$; see Theorem \ref{Jorcx} and Proposition \ref{Torcor}.

\begin{eg} \label{eg6} Let $k$ be a field, $R=k[[x,y,z]]/(x^2,y^2,z^2)$, $M=R/(x,y)$ and $T=R/(y,z)$. Then $R$ is an Artinian complete intersection of codimension three, $M\cong R/xR\tensor_{R} R/yR$, $T\cong R/yR \tensor_{R} R/zR$ and  $M\tensor_{R} R/zR \cong M\otimes_{R} R/zR\cong k$. Moreover the minimal free resolution of $R/yR$ is given by
$\displaystyle{ \textbf{F}=\cdots  \stackrel{y}{\longrightarrow}
R   \stackrel{y}{\longrightarrow} R \stackrel{y}{\longrightarrow} R \to
0 }$. Thus, for a positive integer $i$, we have:
\begin{align} \notag{}
\Tor_{i}^{R}(M,T) = \hh_{i}(M \tensor_{R} T)  & \cong  \hh_{i}((R/xR \tensor_{R} R/yR) \notag{} \tensor_{R}  (R/yR \tensor_{R} R/zR)) \\ \notag{}
& \cong \hh_{i}(R/yR \tensor_{R} (R/xR \tensor_{R}  R/yR \tensor_{R} R/zR)) \\ \notag{}
& \cong \hh_{i}(R/yR \tensor_{R} k)
 \cong  \hh_{i}(\textbf{F} \otimes_{R}  k) \\ \notag{}
& =\hh_{i}( \cdots  \stackrel{0}{\longrightarrow} \notag{}
k  \stackrel{0}{\longrightarrow} k \stackrel{0}{\longrightarrow} k \to
0 )
= k \notag{}
\end{align}

Notice $T^{\ast}=\Hom_{R}(T,R) \cong (yz)R$ and $k \cong k^{\ast} \cong \Tor_{i}^{R}(M,T)^{\ast} \cong \Ext^{i}_{R}(M,T^{\ast})$ for all $i\geq 1$. Hence, setting $N=(yz)R$, we conclude that $\Ext^{i}_{R}(M,N)\cong k$ for all $i\geq 1$. Thus $\cx(M,N)=1$.

Let $\textbf{G}$ and $\textbf{H}$ be the minimal free resolutions of $M$ and $N$, respectively.
One can check that neither $\textbf{G}$ nor $\textbf{H}$ is periodic of period at most two. Therefore $\cx(M)\neq 1$ and $\cx(N)\neq 1$; see \cite[\S6]{Ei}. Moreover $\cx(M)+\cx(N)-3\leq \cx(M,N)\leq \min\{\cx(M),\cx(N)\}$; see \cite[5.7]{AvBu}. So, if either $\cx(M)=3$ or $\cx(N)=3$, then $\cx(M,N)\neq 1$. Consequently $\cx(M)=\cx(N)=2$; see also \cite[9.2]{Coh}.
\end{eg}

\section*{Acknowledgments}
 We are greateful to Greg Piepmeyer for carefully reading the manuscript and for his suggestions that have significantly shortened our proofs. We thank David A. Jorgensen for pointing out Examples \ref{eg5} and \ref{eg6}, and for his comments on an earlier version of this paper. We also thank Petter A. Bergh, Lars W. Christensen and Roger Wiegand for discussions related to this work during its preparation. Our thanks are also due to the anonymous referee for a careful reading and for useful suggestions that have improved the paper.

\end{document}